\documentclass{amsart}
\usepackage{amsthm, amsmath,amsfonts, amssymb, enumerate}
\usepackage{graphicx,tikz}
\usepackage{caption}
\usepackage{subcaption}
\usepackage{hyperref}

\newcommand{\K}{\mathcal K}

\newcommand{\W}{\mathcal W}

\newcommand{\Y}{\mathcal Y}

\newcommand{\Dr}{\mathbb A_{\mathrm{ref}}}
\newcommand{\Da}{\mathbb A_{\mathrm{amb}}}

\newcommand{\acts}{\curvearrowright}

\numberwithin{equation}{section}
\theoremstyle{plain}

\newtheorem{theorem}{Theorem}[section]
\newtheorem*{mainthm*}{Main Theorem}
\newtheorem{lem}[theorem]{Lemma}
\newtheorem{prop}[theorem]{Proposition}
\newtheorem{thm}[theorem]{Theorem}
\newtheorem{cor}[theorem]{Corollary}

\theoremstyle{definition}
\newtheorem{defn}[theorem]{Definition}
\newtheorem{remark}[theorem]{Remark}

\begin{document}
	
\title{A step towards Twist Conjecture}
\author[J.~Huang]{Jingyin Huang}
\address{Math. \& Stats.\\
                    McGill University \\
                    Montreal, Quebec, Canada H3A 0B9}
\email{huangjingyin@gmail.com}
\author[P.~Przytycki]{Piotr Przytycki$^{\dag}$}
\address{Math. \& Stats.\\
                    McGill University \\
                    Montreal, Quebec, Canada H3A 0B9}
\email{piotr.przytycki@mcgill.ca}
\thanks{$\dag$ Partially supported by NSERC, FRQNT, and UMO-2015/\-18/\-M/\-ST1/\-00050.}

\maketitle

\begin{abstract}
\noindent Under the assumption that a defining graph of a Coxeter group admits only twists in $\mathbb{Z}_2$ and is of type FC, we 
prove M\"uhlherr's Twist Conjecture.
\end{abstract}
	
\setcounter{tocdepth}{2}

\section{Introduction}
A \emph{Coxeter generating set} $S$ of a group $W$ is a set such that $(W,S)$ is a Coxeter system. This means that $S$ generates $W$ subject only to relations of the form $s^2=1$ for $s\in S$ and $(st)^{m_{st}}=1$, where $m_{st}=m_{ts}\geq 2$ for $s\neq t\in S$ (possibly there is no relation between $s$ and $t$, and then we put by convention $m_{st}=\infty$).  
An
\emph{$S$-reflection} (or a \emph{reflection}, if the dependence
on $S$ does not need to be emphasised) is an
element of $W$ conjugate to some element of $S$. 
We say that $S$ is \emph{reflection-compatible} with another Coxeter generating set $S'$ if every $S$-reflection is an $S'$-reflection. Furthermore, $S$ is \emph{angle-compatible} with $S'$ if for every $s, t \in S$ with $\langle s,t \rangle$ finite, the set $\{s,t\}$ is conjugate to some
$\{s', t'\} \subset S'$. (Setting $s=t$ shows that angle-compatible implies reflection-compatible.)

M\"uhlherr's Twist Conjecture predicts that angle-compatible Coxeter generating sets of a Coxeter group differ by a sequence of elementary twists. We postpone the definition of an elementary twist to give a brief historical background. For an exhaustive 2006 state of affairs, see \cite{M}.

The Isomorphism Problem for Coxeter groups asks for an algorithm to determine if Coxeter systems $(W,S),(W',S')$ defined by $m_{st},m'_{st}$ give rise to isomorphic groups $W$ and $W'$. Hence listing all Coxeter generating sets $S$ of $W'$ solves the Isomorphism Problem. The articles of Howlett and M\"uhlherr \cite{HM}, and Marquis and M\"uhlherr \cite{MM} reduce the question of listing all such sets $S$ to the problem of listing all $S$ angle-compatible with $S'$. In this way the Twist Conjecture describes a possible solution to the Isomorphism Problem for Coxeter groups.

The first substantial work on the Twist Conjecture is the one by Charney and Davis \cite{CD}, where they prove that if a group acts effectively, properly, and cocompactly on a contractible manifold, then all its Coxeter generating sets are conjugate. Caprace and 
M\"uhlherr \cite{caprace2007reflection} proved that for all $m_{st}<\infty$, a Coxeter generating set $S$ angle-compatible with $S'$ is conjugate to $S'$. This is what was predicted by the Twist Conjecture, since $S$ with all $m_{st}<\infty$ does not admit any elementary twist. Building on that, Caprace and Przytycki \cite{caprace2010twist} proved that an arbitrary $S$ not admitting any elementary twist, and angle-compatible with $S'$, is in fact conjugate to~$S'$. This should be considered as the ``base case'' of the Twist Conjecture.

In a foundational article \cite{muhlherr2002rigidity} M\"uhlherr and Weidmann verified the Twist Conjecture in the case where all  $m_{st}\geq 3$. In that case there occur twists in $\mathbb{Z}_2$ as well as in dihedral groups. Ratcliffe and Tschantz proved the Twist Conjecture for chordal Coxeter groups \cite{RT}.  In these papers the assumptions on $m_{st}$ seem an artefact of the proposed proof. In our paper, we propose the following ``step one'' of a systematic approach towards Twist Conjecture. Our first assumption below is natural from the point of view of the statement of the conjecture, since it says that the occurring elementary twists are as simple as possible. Our second assumption is that $S$ is of \emph{type} FC meaning that 
for any $T\subseteq S$ with $m_{tr}$ finite for all $t,r\in T$, we have that $\langle T \rangle$ is finite. This assumption seems less natural from the point of view of the conjecture statement, but plays a role already in our proof of the ``base case''~\cite{caprace2010twist}.

\begin{mainthm*}
Let $S$ be a Coxeter generating set angle-compatible with $S'$. Suppose that $S$ admits only twists in~$\mathbb{Z}_2$, and is of type $\mathrm{FC}$. Then $S'$ is obtained from~$S$ by a sequence of elementary twists and a conjugation.
\end{mainthm*}

We finally define an elementary twist. Let $(W,S)$ be a Coxeter system. Given a subset $J
\subseteq S$, we denote by $W_J$ the subgroup of $W$ generated by $J$.
We call $J$ \emph{spherical} if $W_J$ is finite. If $J$ is
spherical, let $w_J$ denote the longest element of~$W_J$. We say
that two elements $s\neq t\in S$ are \emph{adjacent} if $\{s,t\}$ is
spherical. This gives rise to a graph whose vertices are $S$ and whose edges (labelled by $m_{st}$) correspond to adjacent pairs of $S$. This graph is called the \emph{defining graph} of $S$. Occasionally, when all $m_{st}$ are finite, we will use another graph, whose vertices are still $S$, but (labelled) edges correspond to pairs of non-commuting elements of $S$. This graph is called the \emph{Coxeter--Dynkin diagram} of $S$. Whenever we talk about adjacency of elements of $S$, we always mean adjacency in the defining graph unless otherwise specified.

Given a subset $J \subseteq S$, we denote by $J^\bot$ the
set of those elements of $S\setminus J$ that commute with $J$. A
subset $J\subseteq S$ is \emph{irreducible} if it is not contained in 
$K\cup K^\bot$ for some non-empty proper subset $K\subset J$.

Let $J\subseteq S$ be an irreducible spherical subset. We say that  $C\subseteq S\setminus (J\cup J^\bot)$ is a \emph{component},
if the subgraph induced on $C$ in the defining graph of $S$ is a connected component of the subgraph induced on $S\setminus (J\cup J^\bot)$. Assume that we have a nontrivial partition
$S\setminus (J\cup J^\bot)=A\sqcup B$, where each component $C$ is contained entirely in $A$ or in $B$. In other words, for all $a \in A$ and $b \in
B$, we have that $a$ and $b$ are non-adjacent. We then say that $J$ \emph{weakly separates} $S$. In the language of groups, this means that $W$ splits as an amalgamated product over
$W_{J \cup J^\bot}$. Note that $A$ and $B$ are in general not
uniquely determined by $J$. 

We then consider the map $\tau \colon S \to W$ defined by
$$
\tau(s)= \left\{\begin{array}{ll}
s & \text{for } s \in A\cup J\cup J^\perp,\\
w_J s w_J^{-1} & \text{for } s \in B,
\end{array}\right.
$$
which is called an \emph{elementary twist in $\langle J \rangle$} (see \cite[Def 4.4]{brady2002rigidity}). 

Coxeter generating sets $S$ and $S'$ of $W$ are \emph{twist equivalent} if $S'$ can be obtained from $S$ by a finite sequence of elementary twists and a conjugation. We say that $S$ is \emph{$k$-rigid} if for each weakly separating $J\subset S$ we have $|J|<k$. Thus $1$-rigid means that there are no elementary twists (this was called \emph{twist-rigid} in \cite{caprace2010twist}). Our Main Theorem states that if a Coxeter generating set $S$ is $2$-rigid, of type FC, and angle-compatible to $S'$, then it is twist equivalent to $S'$. Since twists in $\mathbb{Z}_2$ do not change the defining graph, it follows that $S$ and $S'$ have the same defining graphs. Note that right-angled Coxeter groups are $2$-rigid, and that the Isomorphism Problem for these groups was solved by Radcliffe \cite{R}.

\medskip

\noindent \textbf{Organisation.} 
In the entire article (except for Lemma~\ref{lem:spherical}) we assume that \textbf{$S$ is irreducible, non-spherical, and of type FC}. (The reducible case easily follows from the irreducible.)

In Section~\ref{sec:prem} we recall some basic properties of the Davis complex, and several notions and results from \cite{caprace2010twist}. In Section~\ref{sec:compatibility} we extend in two different ways a marking compatibility result from \cite{caprace2010twist}. Section~\ref{subsec:relative position} contains a technical result required for the definition of complexity used in the proof of the Main Theorem in Section~\ref{sec:proof}.

\medskip

\noindent \textbf{Aknowledgements.} We thank Pierre-Emmanuel Caprace, with whom we designed the strategy carried out in the paper.

\section{Preliminaries}
\label{sec:prem}

\subsection{Davis complex}

Let $\mathbb A$ be the \emph{Davis complex} of a Coxeter system $(W,S)$. The 1-skeleton of $\mathbb A$ is the Cayley graph of $(W,S)$ with vertex set ~$W$ and a single edge spanned on $\{w,ws\}$ for each $w\in W, s\in S$. Higher dimensional cells of $\mathbb A$ are spanned on left cosets in $W$ of remaining finite $W_J$. The left action of $W$ on itself extends to the action on $\mathbb A$.

A \emph{chamber} is a vertex of $\mathbb A$. Collections of chambers corresponding to cosets $wW_J$ are called $J$-\emph{residues} of $\mathbb A$.  A \emph{gallery} is an edge-path in $\mathbb A$. For two chambers $c_1,c_2\in\mathbb A$, we define their \emph{gallery distance}, denoted by $d(c_1,c_2)$, to be the length of a shortest gallery from $c_1$ to $c_2$ (this coincides with the word-metric w.r.t.\ $S$). 

Let $r\in W$ be an $S$-reflection. The fixed point set of the action of $r$ on $\mathbb A$ is called its \emph{wall} $\W_r$. The wall $\W_r$ determines $r$ uniquely. Moreover, $\W_r$ separates~$\mathbb A$ into two connected components, which are called \emph{half-spaces (for $r$)}. 
If a non-empty $K\subset\mathbb A$ is contained in a single half-space (this happens for example if $K$ is connected and disjoint from $\W_r$), then $\Phi(\W_r,K)$ denotes this half-space. An edge of $\mathbb A$ crossed by $\W_r$ is \emph{dual} to~$\W_r$. A chamber is \emph{incident} to $\W_r$ if it is an endpoint of an edge dual to $\W_r$. The \emph{distance} of a chamber $c$ to~$\W_r$, denoted by $d(c,\W_r)$, is the minimal gallery distance from $c$ to a chamber incident to $\W_r$.

The following fact is standard, see eg.\ \cite[Thm 2.9]{building}.
   
\begin{thm}
	\label{lem:residue}
Let $\mathcal R$ be a residue and let $x\in \mathcal R$ and $y\in W$ be chambers.
Then there is a chamber $x'\in\mathcal R$ on a minimal length gallery from $y$ to $x$ such that $\Phi(\W_r, y)=\Phi(\W_r,x')$ for any reflection $r$ fixing $\mathcal R$.
\end{thm}

\subsection{Bases and markings}
\label{subsec:bases and markings}
In this section we recall, in simplified form, several central notions from \cite{caprace2010twist}. 
Let $(W,S)$ be a Coxeter system. Let $\Dr$ be the Davis complex for $(W,S)$ (``ref'' stands for ``reference complex''). For each reflection $r$, let~$\Y_r$ be its wall in $\Dr$. 
The following was called \emph{simple base with spherical support} in \cite{caprace2010twist}.

\begin{defn}\cite[Def 3.1 and 3.6]{caprace2010twist}
	\label{domain} A \emph{base} is a pair
	$(s,w)$ with \emph{core} $s\in S$ and $w\in W$ satisfying
	\begin{enumerate}[(i)]
	\item $w=j_1\cdots j_n$ where $j_i$ are pairwise distinct elements from $S\setminus\{s\}$,
		\item
		$d(w.c_0,\mathcal{Y}_{s})=n$,
		\item
		every wall that separates $w.c_0$ from $c_0$ intersects $\mathcal{Y}_s$, and
		\item
		the \emph{support} $J=\{s,j_1,\ldots,j_n\}$ is spherical.
	\end{enumerate}
	\end{defn}

In \cite[Lem 3.7]{caprace2010twist} and the paragraph preceding it, we established the following.

\begin{remark}
	\label{simple}
	\begin{enumerate}[(i)]
	\item
	If $J\subset S$ is irreducible spherical and $s\in J$, then there
	exists a base with support $J$ and core $s$. Namely, it
	suffices to order the elements of $J\setminus \{s\}$ into a sequence
	$(j_i)$ so that for every $1\leq i\leq n$ the set $\{s, j_1,\ldots,
	j_i\}$ is irreducible. Then $(s,j_1\ldots j_n)$ is a base.
     \item
     The core $s$ and support $J$ determine the base $(s,w)$ uniquely.  Hence we sometimes write a base as $(s,J)$, or even just $J$ if the core is understood. When $J=\{s\}$, we often write $s$ instead of $\{s\}$ for simplicity. 
     \end{enumerate}
     \end{remark}

\begin{defn}
\label{marking} A \emph{marking} is a pair $\mu=((s,J),m)$, where $(s,J)$ is a base and where the \emph{marker} $m\in S$ is not adjacent to some element of $J$. The \emph{core} and the \emph{support} of the marking $\mu$ are the core and the support of its base. 
\end{defn}

Our marking satisfies (but is not equivalent to the marking defined by) \cite[Def~3.8]{caprace2010twist}. To see that, note that by \cite[Rem 3.12]{caprace2010twist}, we have that $w\mathcal Y_m$ is disjoint from $\mathcal Y_s$. 

\begin{remark}
	\label{rem:find markings}
Let $(s,J)$ be a base and $m\in S\setminus (J\cup J^{\perp})$. If $J\cup\{m\}$ is not spherical, then since $S$ is of type FC, the pair $((s,J),m)$ is a marking. In particular, since $S$ is irreducible non-spherical, we have that for each $s\in S$ there exists a marking with core $s$, since we can start with $J\subset S$ maximal irreducible spherical containing~$s$. Similarly, for each $s\in I\subset S$ with $I$ irreducible spherical, there exists a marking with core $s$ and support containing $I$.
\end{remark}

Now suppose that $S$ is reflection-compatible with another Coxeter generating set $S'$. Let $\Da$ be the Davis complex for $(W,S')$ (``amb'' stands for ``ambient complex''). For each reflection $r$, let $\W_r$ be its wall in $\Da$. The following picks up the geometry of the walls $\W_s$ for $s\in S$ inside the ambient complex for $S'$. 

\begin{defn}
	\label{half-space} Let $\mu=((s,w),m)$ be a marking. We define
	$\Phi_s^\mu=\Phi(\mathcal{W}_s, w\mathcal{W}_m)$, which is the half-space for $s$ in $\mathbb A_{\mathrm{amb}}$ containing
	$w\mathcal{W}_m$. 
\end{defn}

\subsection{Geometric set of reflections}
\label{sec:subs}
Let $S,S',W,\Dr$ and $\Da$ be as before, and assume that \textbf{$S$ is angle-compatible with $S'$}. Let $P\subseteq S$.

\begin{defn}
\label{def:geom}
Let $\{\Phi_p\}_{p\in P}$ be a collection of half-spaces of $\Da$ for $p\in P$. The collection $\{\Phi_p\}_{p\in P}$ is \emph{$2$-geometric} if for any pair $p,r\in P$, the set $\Phi_{p}\cap\Phi_{r}\cap \Da^{(0)}$ is a fundamental domain for the action of $\langle p,r\rangle$ on $\Da^{(0)}$. The set $P$ is \emph{$2$-geometric} if there exists a $2$-geometric collection of half-spaces $\{\Phi_p\}_{p\in P}$. \end{defn}

\begin{remark}
\label{geometric}
By \cite[Thm 4.2]{caprace2007reflection}, if $\{\Phi_p\}_{p\in P}$ is $2$-geometric, then after possibly replacing each $\Phi_p$ by opposite half-space, the collection $\{\Phi_p\}_{p\in P}$ is \emph{geometric}, meaning that $F=\bigcap_{p\in P}\Phi_p\cap\Da^{(0)}$ is nonempty. This justifies calling $2$-geometric $P$ \emph{geometric} for simplicity.
In fact, by \cite{Hee} (see also \cite[Thm~1.2]{HRT} and \cite[Fact 1.6]{caprace2007reflection}), if $P$ is geometric, then $F$ is a fundamental domain for the action of $\langle P \rangle$ on $\Da^{(0)}$, and for each $p\in P$ there is a chamber in $F$ incident to $\W_p$. In particular, if $P=S$, then $S$ is conjugate to $S'$. 
\end{remark}

Note that since $S$ is angle-compatible to $S'$, every 2-element subset of $S$ is geometric. However, this does not mean that $S$ is 2-geometric. 
Nevertheless, for $S$ spherical, it is easy to inductively choose 2-geometric $\Phi_s$, and by Remark~\ref{geometric} we obtain the following.

\begin{lem}
	\label{lem:spherical}
	If $S$ is spherical, then it is conjugate to $S'$.
\end{lem}

\begin{cor}
	\label{cor:spherical conjugate}
Let $J\subset S$ be spherical. Then $J$ is conjugate to a spherical $J'\subset S'$. In particular, $J$ is geometric, and if it is irreducible, there exist exactly 2 fundamental domains $F$ for the action of $\langle J \rangle$ on $\Da^{(0)}$ as in Remark~\ref{geometric}.
\end{cor}

\begin{proof}
Let $P\subset S$ be maximal spherical containing $J$. Then $\langle P\rangle$ is a maximal finite subgroup of $W$. By \cite[Thm 1.9]{brady2002rigidity}, we have that $\langle P\rangle$ is conjugate to $\langle P'\rangle$ for a maximal spherical $P'\subset S'$. Thus we can assume without loss of generality that $P=S$ and $P'=S'$. It now suffices to apply Lemma~\ref{lem:spherical}.
\end{proof}

\begin{lem}
	\label{lem:spherical same side}
Let $J\subset S$ be irreducible spherical, and let $F$ be a fundamental domain for $\langle J\rangle$ in $\Da^{(0)}$ guaranteed by Corollary~\ref{cor:spherical conjugate}. Let $s\in J$ and define $w\in W$ via $(s,w)=(s,J)$. Then we have $\Phi(\W_s,F)=\Phi(\W_s,w.F)$.
\end{lem}
\begin{proof}
First suppose $S=S'$. If $c_0\in F$, then by Definition~\ref{domain}(ii) we have $\Phi(\W_s,c_0)=\Phi(\W_s,w.c_0)$, as desired. Otherwise, we have $w_J.c_0\in F$. The half-spaces $\Phi(\W_s,w_J.c_0)$ and $\Phi(\W_s,ww_J.c_0)$ are opposite to $\Phi(\W_s,c_0)$ and $\Phi(\W_s,w.c_0)$, so they coincide as well. 

If $S\neq S'$, then by Corollary~\ref{cor:spherical conjugate} we have $gJg^{-1}=J'$, where $J'$ is a spherical subset of~$S'$. Then $(gsg^{-1},gwg^{-1})$ is a base for $S'$, and by the previous paragraph we have $\Phi(\W_{gsg^{-1}},g.F)=\Phi(\W_{gsg^{-1}},gw.F)$. Translating by $g^{-1}$ we obtain the statement in the lemma.
\end{proof}

The next result is essentially \cite[Prop 5.2]{caprace2010twist}. Except for Lemma~\ref{lem:spherical} this is the only place where we use angle-compatibility (instead of reflection-compatibility). Note that our markings are particular markings of \cite{caprace2010twist}, but the proof of \cite[Prop~5.2]{caprace2010twist} only uses such markings if $S$ is of type FC. 

\begin{prop} 
	\label{consistence}  Suppose that $P\subseteq S$ is irreducible and
	non-spherical.  Let $p_1,p_2\in P$. Suppose that for each $i=1,2$, any marking $\mu$ with core $p_i$ and support and marker in $P$ gives the same
	$\Phi_{p_i}=\Phi_{p_i}^{\mu}$. Then the pair $\{\Phi_{p_1}, \Phi_{p_2}\}$
	is geometric.
\end{prop}

We summarise Remark~\ref{geometric} and Proposition~\ref{consistence} in the following.

\begin{cor}
	\label{cor:geometric criterion}
	If for each $s\in S$ any marking $\mu$ with core $s$ gives rise to the same $\Phi_s^{\mu}$, then $S$ is conjugate to $S'$.
\end{cor}

Also note that since $S$ is of type FC, by \cite[Lem~4.2 and~Thm~4.5]{caprace2010twist} a 1-rigid subset $P\subseteq S$ satisfies the hypothesis of Proposition~\ref{consistence}.

\begin{cor}
\label{thm:twist rigid geometric}
If $P\subseteq S$ is 1-rigid, then it is geometric.
\end{cor}

\section{Compatibility of markings}
\label{sec:compatibility}
Let $S,S',W,\Dr$ and $\Da$ be as in Section~\ref{sec:subs}.

\begin{defn}\cite[Def 4.1]{caprace2010twist}
	\label{def:move}
	Let $((s,J),m), ((s,J'),m')$ be markings with common core. We say that they are
	related by \emph{move}
	\begin{enumerate}
		\item[(M1)] if $J=J'$, and the markers $m$ and $m'$ are adjacent;
		\item[(M2)] if there is $j\in S$ such that $J=J'\cup\{j\}$ and
		moreover $m$ equals $m'$ and is adjacent to $j$.
	\end{enumerate}
We will write $((s,J),m)\sim((s,J'),m')$ if there is a finite sequence of moves of type M1 or M2 that brings  $((s,J),m)$ to $((s,J'),m')$.
\end{defn}

The following is a special case of \cite[Lem 4.2]{caprace2010twist}.
\begin{lem}
	\label{moves do not change half-space} If markings $\mu$ and $\mu'$ with common core $s$
	are related by move $\mathrm{M1}$ or~$\mathrm{M2}$, then $\Phi_s^\mu=\Phi_s^{\mu'}$.
	\end{lem}

The goal of this section is to provide two generalisations of \cite[Thm 4.5]{caprace2010twist}.
\begin{prop}
	\label{lem:compactible1}
Let $I\subset S$ be irreducible spherical. Suppose that no irreducible spherical $I'\supsetneq I$ weakly separates $S$. 
Let $\mu_1=(J_1,m_1)$ and $\mu_2=(J_2,m_2)$ be markings with common core $s\in I$ and such that $I\subseteq J_1,J_2$. Moreover, for $i=1,2$, define $K_i=J_i\setminus (I\cup I^{\perp})$ when $I\subsetneq J_i$, and $K_i=\{m_i\}$ when $J_i=I$. Suppose that $K_1$ and $K_2$ are in the same component $C$ of $S\setminus(I\cup I^{\perp})$. Then $\mu_1\sim\mu_2$.  Consequently $\Phi_s^{\mu_1}=\Phi_s^{\mu_2}$.\end{prop}

\begin{proof}
We follow the proof of Wojtaszczyk \cite[App C]{caprace2010twist}, and argue by contradiction. 
Let $I$ be maximal irreducible spherical satisfying the hypothesis of the proposition
but with $\mu_1\not\sim\mu_2$.

The \emph{$I$-distance} between $\mu_1$ and $\mu_2$ is the length of a shortest edge-path in (the subgraph induced on) $C$ between a vertex of $K_1$ and a vertex of $K_2$. 
(Such a path exists by our hypotheses.)
Among pairs $\mu_1$, $\mu_2$ as above choose a pair with minimal $I$-distance.

If the $I$-distance between $\mu_1$ and $\mu_2$ is $0$, then either $\{m_1\}=K_1=K_2=\{m_2\}$ yielding $\mu_1=\mu_2$, which is a contradiction, or $J_1\cap J_2\setminus {(I\cup I^\perp)}\neq \emptyset$ giving a contradiction with the maximality of $I$. 

Now assume that the $I$-distance between $\mu_1$ and $\mu_2$ is $1$. Then there are two cases to consider. First consider the case where one of $J_i$, say $J_1$, equals $I$.
If also $J_2=I$, then $m_1$ and $m_2$ are adjacent. Thus $\mu_1$ and~$\mu_2$ are related by move M1, which is a contradiction. If $I\subsetneq J_2$, then there exists $k_2\in J_2\setminus (I\cup I^{\perp})$ such that $k_2$ and $m_1$ are adjacent. Thus $\mu_1$ is related to $(I\cup\{k_2\},m_1)$ by move M2. However, $(I\cup\{k_2\},m_1)\sim \mu_2$ by the maximality of $I$, which is a contradiction. It remains to consider the case where $I\subsetneq J_1,J_2$. Then there exist $k_i\in J_i\setminus (I\cup I^{\perp})$ such that $k_1$ and $k_2$ are adjacent. Note that $I\cup\{k_1,k_2\}$ is spherical and irreducible. By Remark~\ref{rem:find markings}, there exists a marking $\nu$ with core $s$ and support containing $I\cup\{k_1,k_2\}$. By the maximality of~$I$, we have $\mu_1\sim \nu\sim \mu_2$, which is a contradiction.

If the $I$-distance between $\mu_1$ and $\mu_2$ is $\ge 2$, let $\gamma$ be a shortest edge-path in~$C$ connecting a vertex $k_1\in K_1$ to a vertex $k_2\in K_2$. Let $k$ be the vertex on~$\gamma$ following~$k_1$. If $I\cup \{k\}$ is spherical, then again by Remark~\ref{rem:find markings}, there exists a marking $\nu$ with core $s$ and support containing $I\cup \{k\}$. Since we chose $\mu_1$ and $\mu_2$ to have minimal $I$-distance, we obtain $\mu_1\sim \nu\sim \mu_2$, which is a contradiction. If $I\cup \{k\}$ is not spherical, then $(I,k)$ is a marking, hence analogously $\mu_1\sim (I,k)\sim \mu_2$, which is a contradiction.
\end{proof}

\begin{prop}
	\label{lem:compactible2}
 Let $P\subseteq S$ be irreducible non-spherical. Suppose that for any irreducible spherical $L\subset S$ with $L\cap P\neq\emptyset$, all elements of $P\setminus(L\cup L^{\perp})$ are in one component of $S\setminus(L\cup L^{\perp})$. Then for any markings $\mu_1$ and $\mu_2$ with supports and markers in $P$ and common core $p$, we have $\mu_1\sim\mu_2$. Consequently $\Phi^{\mu_1}_p=\Phi^{\mu_2}_p$ and by Proposition~\ref{consistence}, $P$ is geometric.
\end{prop}

Note that $P\setminus (L\cup L^{\perp})\neq\emptyset$ for any irreducible spherical $L$. In the proof we will need the following terminology (depending on $P$). A marking $\mu=((p,J), m)$ is \emph{admissible} if \begin{enumerate}
	\item $p\in P$, and
	\item if $L\subset S$ is irreducible spherical such that $p\in L$ and $J\nsubseteq L$, then $J\setminus(L\cup L^{\perp})$ (which is nonempty) and $P\setminus(L\cup L^{\perp})$ are in the same component of $S\setminus(L\cup L^{\perp})$, and
	\item if $L\subset S$ is irreducible spherical such that $J\subseteq L$, then $m$ and $P\setminus(L\cup L^{\perp})$ are in the same component of $S\setminus(L\cup L^{\perp})$.
\end{enumerate}

We also say that a base $(p,J)$ is \emph{admissible} if it satisfies Conditions (1) and (2).

\begin{lem}
	\label{claim}
	Suppose that $(p,J)$ is admissible. Let $\nu=((p,J'),m)$ be such that $J\subseteq J'$, $J'\setminus J\subset P$ and $m\in P$. Then $\nu$ is admissible.
\end{lem}

Note that since $P$ is irreducible non-spherical, such $\nu$ exists for each $J$.

\begin{proof}
Condition (1) is immediate. For Condition (2), pick irreducible spherical $L\subset S$ such that $p\in L$ and $J'\nsubseteq L$. If $J\nsubseteq L$, then $\emptyset\neq J\setminus(L\cup L^{\perp})\subseteq J'\setminus(L\cup L^{\perp})$. Since $(p,J)$ is admissible, Condition (2) holds for such $L$ and $J'$. If $J\subseteq L$, then $J'\setminus(L\cup L^{\perp})\subseteq J'\setminus J\subset P$, hence Condition (2) holds for such $L$ and~$J'$. Condition~(3) is immediate, since we have $m\in P$.
\end{proof}

\begin{proof}[Proof of Proposition~\ref{lem:compactible2}] It is clear that for $p\in P$ the base $(p,\{p\})$ is admissible. Thus by Lemma~\ref{claim} both $\mu_1$ and $\mu_2$ are admissible. 
Hence to prove the proposition it suffices to show that for any two admissible markings $\mu_1,\mu_2$ with common core~$p$, we have $\mu_1\sim \mu_2$. 

We argue by contradiction. Let $I\ni p$ be maximal irreducible spherical such that there are admissible markings $\mu_1=(J_1,m_1)$ and $\mu_2=(J_2,m_2)$ with $I\subseteq J_1,J_2$, and $\mu_1\not\sim \mu_2$. We define $K_1$, $K_2$, and the $I$-distance between $\mu_1$ and $\mu_2$ as in the proof of Proposition~\ref{lem:compactible1}. Since both $\mu_1$ and~$\mu_2$ are admissible, their $I$-distance is finite. Among pairs $\mu_1$, $\mu_2$ as above choose a pair with minimal $I$-distance.

If the $I$-distance is $0$, then either $\mu_1=\mu_2$, or there is irreducible $I'\supsetneq I$ contained in both $J_1$ and $J_2$, contradiction.
Suppose now that the $I$-distance is $1$. There are two cases to consider.

\smallskip\noindent \textbf{Case 1: $J_1=I$}. If $J_2=I$, then $\mu_1$ and $\mu_2$ are related by move M1, contradiction. Now we assume $I\subsetneq J_2$. Pick $k_2\in K_2$ adjacent to $m_1$. Then $I'=I\cup\{k_2\}$ is spherical and irreducible. Moreover, $\mu_1\sim (I',m_1)$ by move M2. We claim that $(I',m_1)$ is admissible. Then $(I',m_1)\sim\mu_2$ by the maximality of $I$, which yields a contradiction. Now we prove the claim. For Condition (2), let $p\in L$ and $I'\nsubseteq L$. If $I\nsubseteq L$, it suffices to use Condition (2) in the admissibility of $\mu_1$. Now suppose $I\subseteq L$. Then $I'\setminus (L\cup L^\perp)=\{k_2\}$. By Condition (2) in the admissibility of $\mu_2$, we have that $k_2$ is in the same component of $S\setminus(L\cup L^{\perp})$ as $P\setminus(L\cup L^{\perp})$, as desired. Condition~(3) follows immediately from Condition (3) in the admissibility of $\mu_1$.

\smallskip \noindent \textbf{Case 2: $I\subsetneq J_1$ and $I\subsetneq J_2$}. For $i=1,2$, pick $k_i\in K_i$ such that $k_1$ and $k_2$ are adjacent. Then $J=I\cup\{k_1,k_2\}$ is spherical and irreducible. It is easy to show that $J$ is admissible following the argument from Case 1. Let $\nu$ be an admissible marking constructed from $J$ as in Lemma~\ref{claim}. Then $\mu_1\sim \nu\sim\mu_2$ by the maximality of $I$, which yields a contradiction. 

\smallskip

Finally suppose that the $I$-distance $d$ between $\mu_1$ and $\mu_2$ is $\ge 2$. Let $\gamma$ be a shortest edge-path in the subgraph induced on $S\setminus (I\cup I^\perp)$ starting at $k_1\in K_1$ and ending at $k_2\in K_2$. Let $k$ be the vertex on $\gamma$ following $k_1$. If $J=I\cup \{k\}$ is not spherical, then let $\nu=(I,k)$, otherwise let $\nu$  be defined from $J$ as in Lemma~\ref{claim}.
Since the $I$-distance between $\nu$ and $\mu_1,\mu_2$ is $< d$, to reach a contradiction it suffices to prove that $\nu$ is admissible.

Consider first the case where $J$ is spherical. By Lemma~\ref{claim}, it suffices to prove that $J$ is admissible. Let $p\in L$ and $J\subsetneq L$. If $I\subsetneq L$, then we use the admissibility of $\mu_1$. Otherwise, we have $J\setminus (L\cup L^\perp)=\{k\}$. Since $\gamma$ is a geodesic, $\gamma\cap L$ is empty, a vertex, or an edge. Thus there is a subpath of $\gamma$  from $k$ to $k_1$ or $k_2$ outside $L$. Since $\mu_1,\mu_2$ were admissible, $k$ is in the component of $S\setminus (L\cup L^\perp)$ containing $P\setminus (L\cup L^\perp)$, as desired. The case where $J$ is not spherical is similar.
\end{proof}

\section{Relative position of maximal spherical subsets} 
\label{subsec:relative position}
In this section, we introduce particular subsets of pairs of maximal spherical residues (which will be in Section~\ref{sec:proof} involved in the definition of the complexity of a Coxeter generating set with respect to another one). It is crucial to prove that these subsets are well-defined (Proposition~\ref{prop:consistent}) which is the most technical part of the article, and we recommend to skip it at first reading. Let $S,S',W,\Dr$ and $\Da$ be as in Section~\ref{sec:subs}. \textbf{Throughout the remaining part of the article, we will also assume that $S$ is 2-rigid}.

Let $J\subset S$ be a maximal spherical subset. By Corollary~\ref{cor:spherical conjugate}, $W_J$ stabilises a unique maximal cell $\sigma_J\subset\Da$. Let $C_J$ be the collection of vertices in this maximal cell and let $D_J$ be the elements of $C_J$ incident to each $\W_j$ for $j\in J$. When $J$ is irreducible, then by Corollary~\ref{cor:spherical conjugate}, it is easy to see that $D_J$ is made of two antipodal vertices. In general, let $J=J_1\sqcup\cdots\sqcup J_k$ be the decomposition of $J$ into maximal irreducible subsets. Let $\sigma_J=\sigma_1\times\cdots\times\sigma_k$ be the induced product decomposition of the associated cell. Then $D_J$ is a product of pairs of antipodal vertices $\{u_i,v_i\}$ for each $\sigma_i$. Let $\pi_i\colon D_J\rightarrow \{u_i,v_i\}$ be the coordinate projections.

\begin{defn} 
\label{def:good}
Let $J_1\subset S$ be irreducible spherical and $r\in S$. A vertex $t\in J_1$ is \emph{good with respect to $r$}, if $t$ is adjacent to $r$, or $J_1\setminus (t\cup t^{\perp})$ is non-empty and in the same component of $S\setminus(t\cup t^{\perp})$ as $r$. Note that being good depends on $J_1$.

Let $J$ and $I$ be two maximal spherical subsets of $S$. A maximal irreducible subset~$J_1$ of~$J$ is \emph{good} with respect to $I$ if there exist  non-adjacent $t\in J_1$ and $r\in I$ such that $t$ is good with respect to $r$.
\end{defn}

\begin{defn}
	\label{def:E}
For each ordered pair $(J,I)$ of maximal spherical subsets of $S$, we define the following subset $E_{J,I}\subseteq D_J$. First, for each $i$, consider the following $E^i_{J,I} \subseteq D_J$. If $J_i$ is not good with respect to $I$, then we take $E^i_{J,I}=D_J$. If $J_i$ is good, then let $t$ and $r$ be as in Definition~\ref{def:good}. Then we take $E^i_{J,I}=D_J\cap \Phi(\W_t,\W_r)$ (which is $\pi^{-1}_i(u_i)$ or $\pi_i^{-1}(v_i)$). We define $E_{J,I}=E^1_{J,I}\cap\cdots\cap E^k_{J,I}$. \end{defn}

The goal of this section is to prove the following, saying that $E^i_{J,I}$ does not depend on the choice of $t$ and $r$. 

\begin{prop}
	\label{prop:consistent}
Let $J_1,J$ and $I$ be as Definition~\ref{def:good}. Suppose that we have pairs of non-adjacent vertices $(t,r)$ and $(t',r')$ in $J_1\times I$ such that $t$ is good with respect to $r$, and $t'$ is good with respect to $r'$. Then $E^1_{J,I}={E^1_{J,I}}'$, where
${E^1_{J,I}}'=D_J\cap \Phi(\W_{t'},\W_{r'})$.
\end{prop}

We need some preparatory lemmas.

\begin{lem}
	\label{double}
Let $s,t\in S$ be adjacent, and let $r\in S$ be neither adjacent to $s$ nor to~$t$. Suppose that $r$ and $s$ are in distinct components of $S\setminus(t\cup t^{\perp})$, and that $r$ and~$t$ are in distinct components of $S\setminus(s\cup s^{\perp})$ (in particular $s$ and $t$ do not commute). Let $J=\{s,t\}$. Then each point in (the unique component) $S\setminus(J\cup J^{\perp})$ is neither adjacent to $s$ nor to $t$.
\end{lem}

\begin{proof}
	Suppose that the collection of vertices of $S\setminus(J\cup J^{\perp})$ that are adjacent to~$s$ or to $t$ is non-empty. Since $S$ is 2-rigid, there is a shortest edge-path $\gamma$ in the subgraph induced on $S\setminus(J\cup J^{\perp})$ that connects $r$ to a vertex $p\in S\setminus(J\cup J^{\perp})$ adjacent to $s$ or $t$. We assume without loss of generality that $p$ is adjacent to $t$. Since $r$ and $t$ are in distinct components of $S\setminus(s\cup s^{\perp})$, there is a vertex $p'$ of $\gamma$ in~$s^{\perp}$. If $p\neq p'$, then the subpath $\gamma'\subseteq\gamma$ from $r$ to $p'$ is a shorter path from $r$ to a vertex adjacent to $s$ or $t$, which is a contradiction. If $p=p'$, then since $r$ and~$s$ are in distinct components of $S\setminus(t\cup t^{\perp})$, there exists a vertex $p''$ of $\gamma'=\gamma$ in $t^{\perp}$. If $p''\neq p$, then we can reach a contradiction as before. If $p''=p$, then $p\in J^{\perp}$, which is impossible by our choice of $\gamma$. 
\end{proof}

\begin{lem}
\label{lem:same side}
	
Let $t,r\in S$ be non-adjacent. Let $J\subset S$ be maximal spherical containing $t$ and let $J_1$ be the maximal irreducible subset of $J$ containing $t$. Let $j_0\in J_1$ and let $\omega=(j_0,j_1,\ldots)$ be the geodesic edge-path in the Coxeter--Dynkin diagram of $J_1$ that starts at $j_0$ and ends at $t$ (such a geodesic is unique since the Coxeter--Dynkin diagram of a spherical subset is a tree). Let $j_n$ be the first vertex of $\omega$ not adjacent to $r$ (possibly $j_n=j_0$ or $j_n=t$). Suppose that both $t,j_0\in J_1$ are good w.r.t.\ $r$. Then we have $\Phi(j_nj_{n-1}\cdots j_1\W_{j_0},E^1_{J,I})=\Phi(j_nj_{n-1}\cdots j_1\W_{j_0},\W_r)$.
\end{lem}

\begin{proof}
We write $E=E^1_{J,I}$ to shorten the notation.

We claim that for any non-commuting $j,j'\in J_1$ at least one of $j,j'$ is good (w.r.t.\ $r$; we will skip repeating this in this proof). To justify the claim, if both $j$ and $j'$ are not good, then 
$r$ and $j$ are in distinct components of $S\setminus(j'\cup {j'}^{\perp})$, and $r$ and $j'$ are in distinct components of $S\setminus(j\cup j^{\perp})$. If $\{j,j'\}\subsetneq J_1$, then there is an element in $S\setminus(\{j,j'\}\cup\{j,j'\}^{\perp})$ adjacent to $j$ and $j'$, which contradicts Lemma~\ref{double}. If $\{j,j'\}=J_1$, then one of $j,j'$ equals $t$, which was assumed to be good, contradiction. This justifies the claim.
	
If $j_0=t$, then there is nothing to prove. Otherwise, we induct on the length of~$\omega$ and assume that the conclusion of the lemma holds for all good $j_i$ distinct from~$j_0$.  By the claim either $j_1$ or $j_2$ is good. We look first at the situation where $j_1$ is good. There are four cases to consider.
	
\smallskip\noindent \textbf{Case 1: both $j_0$ and $j_1$ are not adjacent to $r$.} Since both $j_0$ and $j_1$ are good, we deduce from Proposition~\ref{lem:compactible1} and the assumption that $S$ is 2-rigid that $(j_1,r)\sim((j_1,j_0),r)$ and $(j_0,r)\sim((j_0,j_1),r)$. Let $\Sigma\subset \Da$ be the union of the two sectors of the form $\Phi_{j_0}\cap \Phi_{j_1}$ for $\{\Phi_{j_0}, \Phi_{j_1}\}$ geometric.
Then $(j_1,r)\sim((j_1,j_0),r)$ implies $\W_r\subset \Sigma\cup j_0\Sigma$ and $(j_0,r)\sim((j_0,j_1),r)$ implies $\W_r\subset \Sigma\cup j_1\Sigma$. Thus $\W_r\subset \Sigma$. By induction assumption, $\Phi(\W_{j_1},E)=\Phi(\W_{j_1},\W_r)$, thus $E$ and $\W_r$ are in the same sector of $\Sigma$, and it follows that $\Phi(\W_{j_0},E)=\Phi(\W_{j_0},\W_r)$.
	
\smallskip\noindent \textbf{Case 2: $j_1$ is adjacent to $r$, but $j_0$ is not adjacent to $r$.} Then $n=0$. Let $j_m$ be the first vertex of $\omega$ distinct from $j_0$ not adjacent to $r$. 
 
First, we claim $\Phi(j_0\W_{j_1},E)=\Phi(j_0\W_{j_1},\W_r)$. Indeed, since  $(j_1,j_0)$ and $(j_1,j_2\cdots j_m)$ are bases, by two applications of Lemma~\ref{lem:spherical same side} we have
$$ \Phi(\W_{j_1},j_0.E)=\Phi(\W_{j_1}, E)=\Phi(\W_{j_1},j_2\cdots j_m. E),$$
which equals $\Phi(\W_{j_1},j_2\cdots j_m\W_r)$ by induction. Furthermore, $((j_1,j_2\cdots j_m),r)$ is a marking and $j_2$ is adjacent to $j_0$. Thus by Proposition~\ref{lem:compactible1} and the fact that $S$ is 2-rigid, we obtain
$$ ((j_1,j_2\cdots j_m),r) \sim ((j_1,j_0),r), $$
and the claim follows. 

Let $\Phi_{j_0}, \Phi_{j_1}$ be the half-spaces for $j_0,j_1$ containing $E$ and let $\Lambda=\Phi_{j_0}\cap\Phi_{j_1}$. Since $\W_r$ intersects $\W_{j_1}$, by the claim we have that $\W_r$ intersects $\Lambda$.
It follows that $\Phi(\W_{j_0},E)=\Phi(\W_{j_0},\W_r)$.
	
\smallskip\noindent \textbf{Case 3: $j_0$ is adjacent to $r$, but $j_1$ is not adjacent to $r$.} By induction, we have $\Phi(\W_{j_1},E)=\Phi(\W_{j_1},\W_r)$. We need to show $\Phi(j_1\W_{j_0},E)=\Phi(j_1\W_{j_0},\W_r)$. To do this, it suffices to reverse the argument in the previous paragraph.
	
\smallskip\noindent \textbf{Case 4: both $j_0$ and $j_1$ are adjacent to $r$.} Let $P=\{j_0,j_1,\ldots,j_n,r\}$. We claim that $P$ is geometric. Indeed, since $P$ is irreducible and non-spherical, to justify the claim it suffices to verify the hypothesis of Proposition~\ref{lem:compactible2}. Let $L\subset S$ be irreducible spherical with $L\cap P\neq\emptyset$. Since $S$ is 2-rigid, it suffices to consider $L=\{l\}$ a singleton in $P$. Note that in $P$ the only two non-adjacent elements are $r$ and $j_n$. Thus the cases $l=r,j_n$ are clear. It remains to consider the case $l\in K=P\setminus \{r,j_n\}$. Since $K$ is irreducible and $|K|\geq 2$, we have $K\setminus (l\cup l^{\perp})\neq\emptyset$. Consequently, $l$ does not weakly separate $P$, verifying the claim.
	
By Remark~\ref{geometric}, there are half-spaces $\{\Phi_{j_0},\Phi_{j_1},\cdots,\Phi_{j_n},\Phi_r\}$ whose intersection contains a vertex $x$ incident to $\W_r$. Thus by induction we have $$\Phi(j_nj_{n-1}\cdots j_2\W_{j_1},E)=\Phi(j_nj_{n-1}\cdots j_2\W_{j_1},\W_r)=\Phi(j_nj_{n-1}\cdots j_2\W_{j_1},x).$$ Let $F$ and $F_{\mathrm{ant}}$ (resp.\ $V$ and $V_{\mathrm{ant}}$) be the two fundamental domains for $\{j_0,j_1,\ldots,j_n\}$ (resp. $\{j_1,\ldots,j_n\}$) from Corollary~\ref{cor:spherical conjugate}. Assume without loss of generality $F\subset V$. Then $x$ and $E$ are both inside $F$ or $F_{\mathrm{ant}}$, say $F$, otherwise they would be separated by $j_nj_{n-1}\cdots j_2\W_{j_1}$. It follows that both $x$ and $E$ are in $V$. In particular, $\Phi(j_nj_{n-1}\cdots j_1\W_{j_0},E)=\Phi(j_nj_{n-1}\cdots j_1\W_{j_0},x)$, which equals $\Phi(j_nj_{n-1}\cdots j_1\W_{j_0},\W_r)$, as desired.
	
Now we turn to the situation where $j_1$ is not good, hence $j_2$ is good. Since $j_1$ is not good, it is not adjacent to $r$,
and furthermore $r$ is not adjacent to $j_0$ or~$j_2$. Since $j_0$ is good and $S$ is 2-rigid, by 
Proposition~\ref{lem:compactible1} we obtain $\Phi(\W_{j_0},\W_r)=\Phi(\W_{j_0},j_1\W_r)=\Phi(\W_{j_0},j_1j_2\W_r)$. Similarly, $\Phi(\W_{j_2},\W_r)=\Phi(\W_{j_2},j_1\W_r)=\Phi(\W_{j_2},j_1j_0\W_r)$. Since $\{j_0,j_1,j_2\}$ is conjugate to a subset of $S'$ by Corollary~\ref{cor:spherical conjugate}, we deduce that $\Phi(\W_{j_0},E)=\Phi(\W_{j_0},\W_r)$ by applying Lemma~\ref{lem:3-generator} below with $s_1=j_0$, $s_2=j_1$ and $s_3=j_2$.
\end{proof}

\begin{figure}
\includegraphics[scale=0.68]{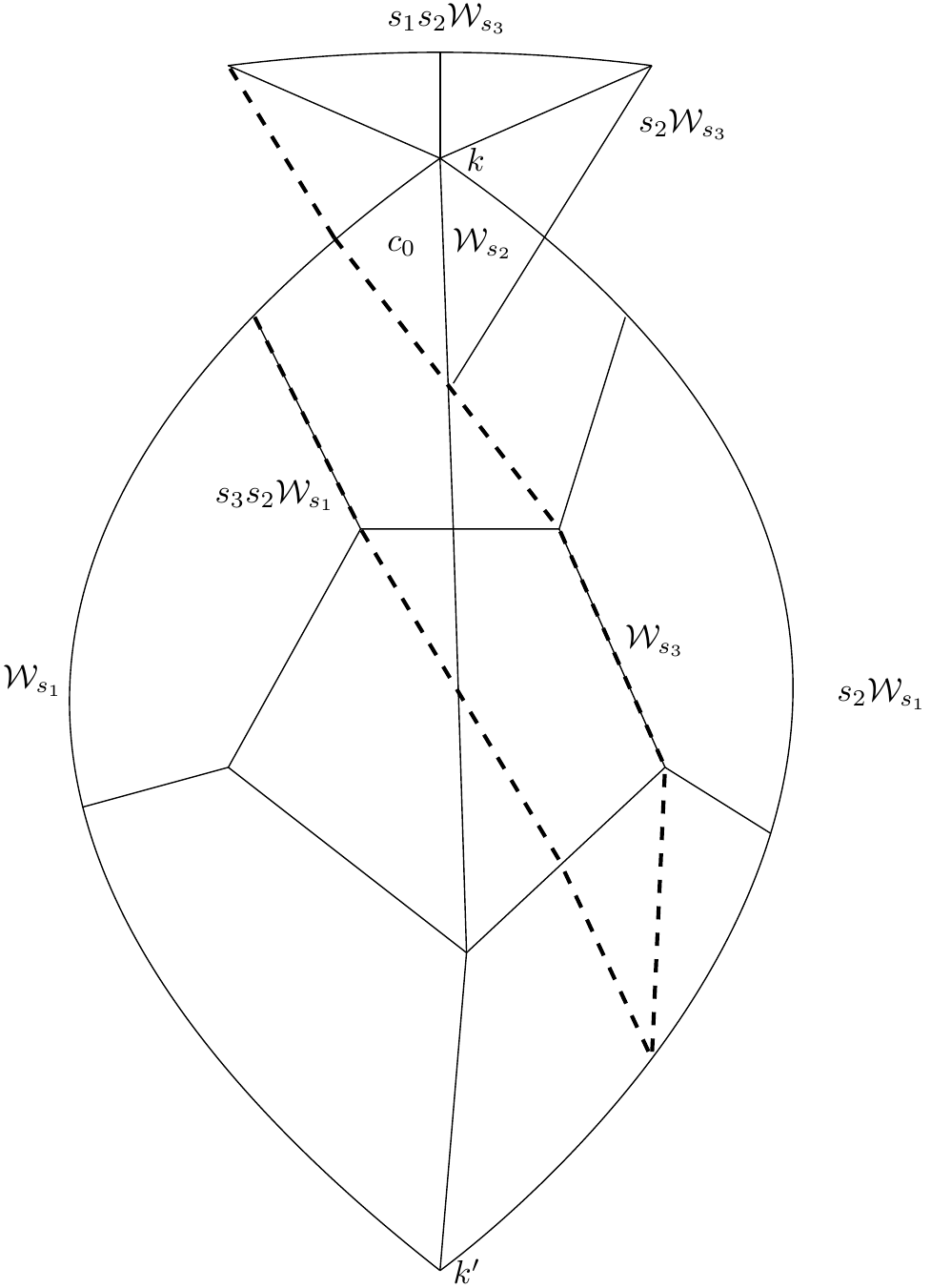}	
\caption{$(2,3,5)$-triangle group}
\label{fig1}
\end{figure}

\begin{lem}
	\label{lem:3-generator}
Let $H$ be a 3-generator irreducible spherical Coxeter group and let $\sigma$ be the associated Davis cell. Let $s_1,s_2,s_3$ be consecutive vertices in the Coxeter--Dynkin diagram of $H$, and let $\W_{s_1}$, $\W_{s_2}$ and $\W_{s_3}$ be the associated walls of $\sigma$. Let $c_0$ be a chamber of $\sigma$ that is incident to each of $\W_{s_i}$ for $1\le i\le 3$. Let $c$ be an arbitrary chamber satisfying all of the following.
\begin{enumerate}
	\item $\Phi(\W_{s_1},c)=\Phi(\W_{s_1},c_0)$;
	\item $\Phi(\W_{s_1},c)=\Phi(\W_{s_1},s_2c)=\Phi(\W_{s_1},s_2s_3c)$;
	\item $\Phi(\W_{s_3},c)=\Phi(\W_{s_3},s_2c)=\Phi(\W_{s_3},s_2s_1c)$.
\end{enumerate}
Then $\Phi(\W_{s_3},c)=\Phi(\W_{s_3},c_0)$.
\end{lem}

\begin{figure}
	\centering
	\begin{subfigure}{.5\textwidth}
	    \centering
	 	\includegraphics[scale=0.6]{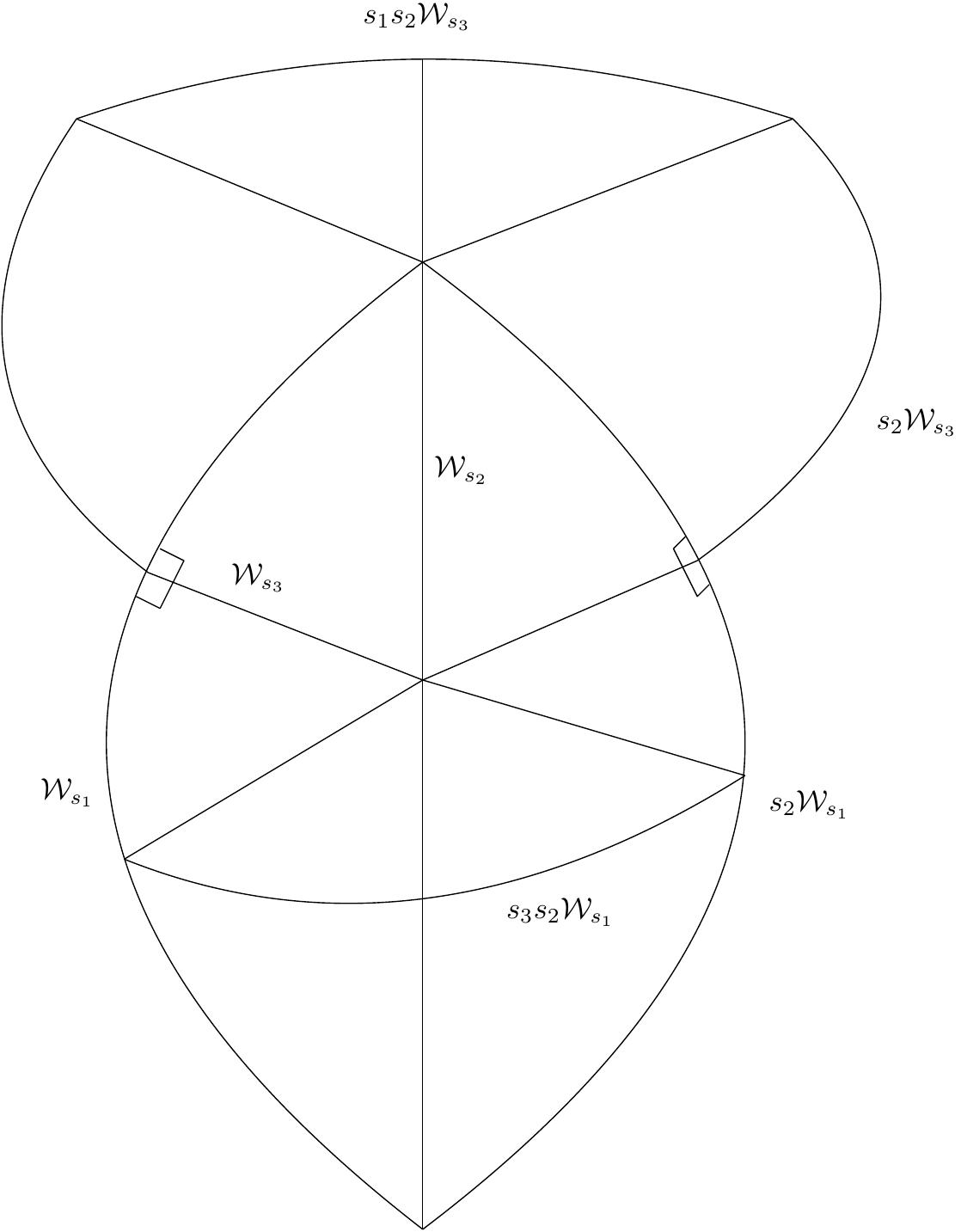}
	 	\caption{$(2,3,3)$}
		 \label{fig2}
	\end{subfigure}%
	\begin{subfigure}{.5\textwidth}
	   	\centering
		\includegraphics[scale=0.6]{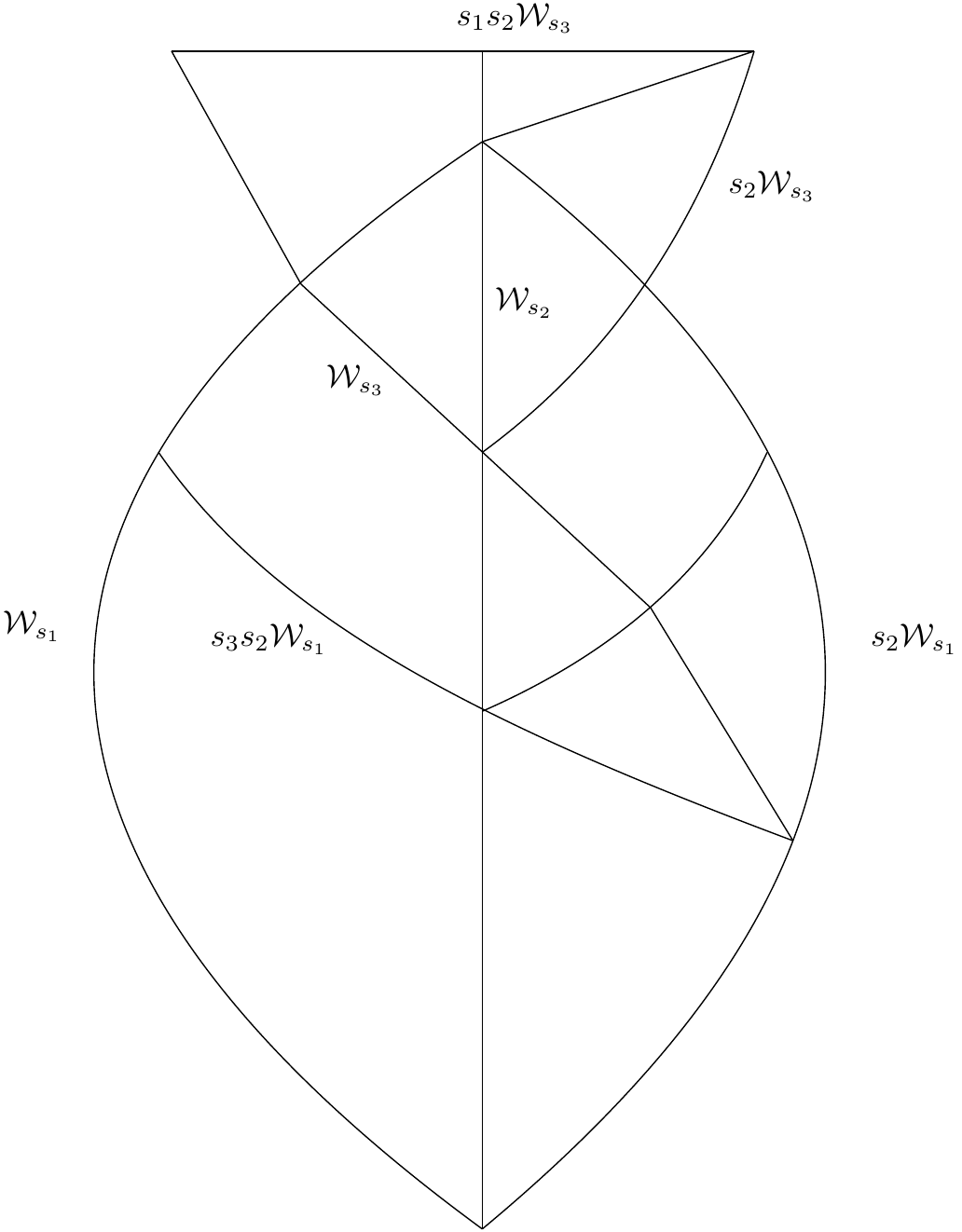}
		\caption{$(2,3,4)$}
		 \label{fig3}
	\end{subfigure}
\end{figure}

\begin{proof}
We will prove that either $c=c_0$, or $c=s_2.c_0$, which implies immediately the lemma. We first consider the case where $H$ is the $(2,3,5)$-triangle group. Assume first $(s_2s_3)^5=1$. Consider the tilling of the regular dodecahedron obtained from drawing all the walls of $H$. A direct computation gives 
Figure~\ref{fig1}. It follows from Conditions (1) and (2) that $c\in \Phi(\W_{s_1},c_0)\cap s_2\Phi(\W_{s_1},c_0)\cap s_3s_2\Phi(\W_{s_1},c_0)$. In other words, $c$ is inside the region $R_1$ bounded by $\W_{s_1}$, $s_2\W_{s_1}$ and $s_3s_2\W_{s_1}$ containing~$c_0$. Similarly, by Condition (3), $c$ is inside the region $R_2$ bounded by $\W_{s_3}$, $s_2\W_{s_3}$ and $s_1s_2\W_{s_3}$ that contains either $c_0$ or its antipodal chamber. 
However, the latter case is impossible since $R_1\cap R_2$ contains $c$ and is thus non-empty. Thus $R_1\cap R_2$ is the union of the two triangles adjacent along $\W_{s_2}$, one of which contains $c_0$, as desired. 

If $(s_2s_3)^3=1$, the proof is analogous. The case of the $(2,3,3)$-triangle group and the $(2,3,4)$-triangle group can be proved in a similar way, see Figures~\ref{fig2} and~\ref{fig3} to chase down the relevant regions.
\end{proof}

We are finally ready for the following.

\begin{proof}[Proof of Proposition~\ref{prop:consistent}]
We prove the proposition by induction on the distance between $t$ and $t'$ in the Coxeter--Dynkin diagram of $J_1$. If $t=t'$, then since $\W_{r}\cap\W_{r'}\neq\emptyset$, the proposition is clear. Also note that if $r=r'$, then the proposition follows from Lemma~\ref{lem:same side}. 

Now we assume $t\neq t'$ and $r\neq r'$. If $t$ and~$r'$ are non-adjacent, then $t$ is good with respect to $r'$ (since $r$ and $r'$ are adjacent). Thus we can pass from $(t,r)$ to $(t',r')$ via $(t,r')$ by the previous discussion. The case where $t'$ and~$r$ are non-adjacent is analogous. Thus it remains to consider the case where $t$ and~$r'$ are adjacent, and $t'$ and $r$ are adjacent.

We first look at the distance one case: where $t$ and $t'$ do not commute. We consider $P=\{t,t',r,r'\}$. Note that the defining graph of $P$ is a square, thus $P$ is 1-rigid. Hence $P$ is geometric by Corollary~\ref{thm:twist rigid geometric}. Let $F\subset \Da^{(0)}$ be the fundamental domain for $\langle P\rangle\acts \Da^{(0)}$ from Remark~\ref{geometric}. Let $V\subset \Da^{(0)}$ be the fundamental domain for $\langle t,t'\rangle$ that contains $F$. Since $t$ and $t'$ do not commute, $V$ is the only fundamental domain for $\langle t,t'\rangle$ contained in $\Phi(\W_{t},\W_{r})$ and the only one in $\Phi(\W_{t'},\W_{r'})$. Thus $E^1_{J,I}\subset V$ and ${E^1_{J,I}}'\subset V$. It follows that $E^1_{J,I}={E^1_{J,I}}'$.

Now we deal with the general situation. We consider the geodesic edge-path $(t_i)_{i=0}^n$ from $t_0=t$ to $t_n=t'$ in the Coxeter--Dynkin diagram of $J_1$ (which is a tree). Let $i'$ be minimal such that $t_{i'}$ is not adjacent to $r'$ and $i$ maximal such that $t_{i}$ is not adjacent to $r$. Then $t_{i'}$ is good respect to $r'$ and $t_{i}$ is good with respect to~$r$. Note that $i'\ge 1$ and $i\le n-1$. If $i'\le n-1$, then by the induction assumption we can pass from $(t,r)$ to $(t',r')$ via $(t_{i'},r')$. The case $i\geq 1$ is analogous. Thus in the remaining part of the proof we assume $i'=n$ and $i=0$, in other words, $t_i$ is adjacent to both $r$ and $r'$ for each $1\le i\le n-1$. 

Let $P=\{t_0,\ldots,t_n,r,r'\}$. Note that the defining graph of $P$ is a join of a 4-cycle (whose consecutive vertices are $t,r',r,t'$) and a complete graph (whose vertices are $t_1,\ldots,t_{n-1}$). Since $(t_i)$ was an edge-path in the Coxeter--Dynkin diagram, it is easy to prove that the defining graph of $P$ is 1-rigid. Thus $P$ is geometric by Corollary~\ref{thm:twist rigid geometric}. Let $F\subset \Da^{(0)}$ be the fundamental domain for $\langle P\rangle\acts \Da^{(0)}$ from Remark~\ref{geometric}. Let $V\subset \Da^{(0)}$ be the fundamental domain for $\langle t_0,\ldots,t_n\rangle$ that contains $F$. Since $\{t_0,\ldots,t_n\}$ is irreducible, $V$ is the only fundamental domain for $\langle t_0,\ldots,t_n\rangle$ contained in $\Phi(\W_{t},\W_{r})$ and the only one in $\Phi(\W_{t'},\W_{r'})$. Thus $E^1_{J,I}\subset V$ and ${E^1_{J,I}}'\subset V$. Hence $E^1_{J,I}={E^1_{J,I}}'$.\end{proof}

\section{Proof of the main theorem}
\label{sec:proof}
We keep the setup from Section~\ref{subsec:relative position}.

\begin{defn}
	\label{def:complexity function}
We define the \emph{complexity} of $S$, denoted $\K(S)$, to be an ordered pair of numbers 
	\begin{center}
		$\big(\K_1(S),\K_2(S)\big)=\Big(\sum_{J\neq I}d(C_J,C_{I}),\sum_{J\neq I}d(E_{J,I},E_{I,J})\Big)$,
	\end{center}
where $J$ and $I$ range over all maximal spherical subsets of $S$, and $E_{J,I}$ is defined in Definition \ref{def:E}. Note that the distance is computed in $\Da$ and so we have $\K_1(S')=\K_2(S')=0$.
	
For two Coxeter generating sets $S$ and $S_\tau$, we define $\K(S_\tau)<\K(S)$ if $\K_1(S_\tau)<\K_1(S)$, or $\K_1(S_\tau)=\K_1(S)$ and $\K_2(S_\tau)<\K_2(S)$.
\end{defn}

Note that since $S$ is 2-rigid, an elementary twist does not change its defining graph. 
Thus Main Theorem reduces to the following.

\begin{thm}
\label{thm:minimal complexity}
Let $S$ be angle-compatible with $S'$. Suppose that $S$ is 2-rigid and of type~$\mathrm{FC}$. Assume moreover that $S$ has minimal complexity among all Coxeter generating sets twist-equivalent to $S$. Then $S$ is conjugate to $S'$.
\end{thm}

The proof will take the remaining part of the article, and we divide it into several steps.
For $\mu=((s,w),m)$ a marking with support $J$, we define $K_{\mu}=J\setminus\{s\}$ if $J\neq\{s\}$, and $K_\mu=\{m\}$ otherwise.

By Corollary~\ref{cor:geometric criterion}, to prove Theorem~\ref{thm:minimal complexity} it suffices to show that for any markings $\mu$ and $\mu'$ with common core $s\in S$, we have $\Phi^{\mu}_s=\Phi^{\mu'}_s$. Note that for each component $A$ of $S\setminus(s\cup s^{\perp})$, there exists a marking $\mu$ with $K_\mu\subseteq A$. By Proposition~\ref{lem:compactible1} and the fact that $S$ is 2-rigid, if $K_{\mu'}\subseteq A$, then $\Phi^{\mu}_s=\Phi^{\mu'}_s$. Thus each component $A$ of $S\setminus(s\cup s^{\perp})$ determines a half-space $\Phi_A:=\Phi^{\mu}_s$ for $s$. Two components $A_1$ and $A_2$ of $S\setminus(s\cup s^{\perp})$ are \emph{compatible} if $\Phi_{A_1}=\Phi_{A_2}$. We will show that all the components of $S\setminus(s\cup s^{\perp})$ are compatible. Fixing $s\in S$, we shall divide these components into several classes and conduct a case analysis.

\subsection{Big components are compatible}

\begin{defn}
 A component $A$ of $S\setminus(s\cup s^{\perp})$ is \emph{big} if there is $a\in A$ not adjacent to $s$. Otherwise $A$ is \emph{small}.
\end{defn}

\begin{lem}
	\label{lem:big comp}
Any two big components are compatible.
\end{lem}

\begin{proof}
We argue by contradiction and assume that the big components of $S\setminus(s\cup s^{\perp})$ can be divided into two non-empty families $\{A_k\}$ and $\{B_k\}$ such that all $\Phi_{A_k}$ coincide (call that half-space $\Phi_A$) and are distinct from all $\Phi_{B_k}$, which also coincide (call that half-space $\Phi_B$). Let $B$ be the union of all the $B_k$. Let $\tau$ be the elementary twist that sends each element $b\in B$ to $sbs$ and fixes other elements of $S$. For a contradiction, we will prove $\K_1(\tau(S))<\K_1(S)$.

Let $J\subset S$ be maximal spherical. $J$ is \emph{twisted} if it contains an element of~$B$ and $s\notin J$. A twisted $J$ exists, since we can take any maximal spherical $J$ containing $b\in B$ not adjacent to $s$. Note that if $J$ is twisted, then for each $j\in J$ we have $\W_{\tau(j)}=s\W_j$, and hence $C_{\tau(J)}=s.C_{J}$. Moreover, there is an element $b\in J\setminus\{s\}$ not adjacent to $s$, since otherwise $J\cup\{s\}$ would be spherical contradicting the maximality of $J$. Then $\Phi(\W_s,C_J)=\Phi(\W_s,\W_b)=\Phi_B$. 

Consider now maximal spherical $I\subset S$ that is not twisted. If $s\in I$, then $C_{\tau(I)}=s.C_I=C_{I}$. If $s\notin I$, then $I\cap B=\emptyset$, and we also have $C_{\tau(I)}=C_{I}$. As before, there exists such $I$ with $s\notin I$. Moreover, then there is $a\in I\setminus\{s\}$ not adjacent to $s$, and  $\Phi(\W_s,C_I)=\Phi(\W_s,\W_a)=\Phi_A$. 

Let $J,I\subset S$ be maximal spherical. If both $J$ and $I$ are twisted or both are not twisted, then $d(C_{J},C_{I})=d(C_{\tau(J)},C_{\tau(I)})$. Now suppose that $J$ is twisted and $I$ is not twisted. If $s\in I$, we still have $d(C_{J},C_{I})=d(C_{\tau(J)},C_{\tau(I)})$. If $s\notin I$, then since $\Phi_B\neq\Phi_A$, we have  $\Phi(\W_s,C_J)\neq \Phi(\W_s,C_I)$. Hence a minimal length gallery~$\beta$ from a chamber in $C_J$ to a chamber in $C_I$ has an edge dual to $\W_s$. Removing this edge from $\beta$ and reflecting $\beta\cap\Phi(\W_s,C_J)$ by $s$, we obtain a shorter gallery from a chamber in $s.C_J$ to a chamber in $C_I$.
Thus $d(C_{\tau(J)},C_{\tau(I)})=d(s.C_{J},C_{I})<d(C_{J},C_{I})$. Consequently $\K_1(\tau(S))<\K_1(S)$.
\end{proof}

\subsection{Exposed components}

\begin{defn}
	\label{defn:type I}
A small component $A$ is \emph{exposed} if there is $t\in A$ and $r$ inside a different component of $S\setminus(s\cup s^{\perp})$ such that 
$s$ and $r$ are in distinct components of $S\setminus(t\cup t^\perp)$.
\end{defn}

\begin{lem}
\label{lem:small comp}
If there exists an exposed component, then all components are compatible.
\end{lem}

\begin{proof}
Let $t$ and $r$ be as in Definition~\ref{defn:type I}. Note that $r$ is adjacent to neither $s$ nor~$t$. By Lemma~\ref{double}, none of the elements of $S\setminus(\{s,t\}\cup\{s,t\}^{\perp})$ is adjacent to $s$ or~$t$. It follows that there is only one small component of $S\setminus(s\cup s^{\perp})$, and this small component equals $\{t\}$.

Observe that a maximal spherical subset $J\subset S$ contains $s$ if and only if it contains $t$. Indeed, if say $s\in J$, then each element of $J\setminus\{s\}$ is adjacent to $s$. Hence $J\subseteq \{s,t\}\cup\{s,t\}^{\perp}$ by Lemma~\ref{double}. If $t\notin J$, then $J\cup \{t\}$ is spherical, which contradicts the maximality of $J$. We say that $J$ is \emph{exposed} if $\{s,t\}\subseteq J$.

Let $\W_{\{s,t\}}$ be the union of all the walls in $\Da$ for the reflections in the dihedral group $\langle s,t\rangle$. 
Since $S$ is $2$-rigid, the graph induced on $S\setminus(\{s,t\}\cup\{s,t\}^{\perp})$ is connected. Thus all the walls $\W_r$ for $r\in S\setminus(\{s,t\}\cup\{s,t\}^{\perp})$ lie in the same connected component $\Lambda$ of $\Da\setminus \W_{\{s,t\}}$. Consequently, all $D_J$ for $J$ not exposed lie in $\Lambda$. Let $\Sigma\subset \Da$ be the union of the two sectors of the form $\Phi_s\cap \Phi_{t}$ for $\{\Phi_s, \Phi_{t}\}$ geometric. Assume first $\Lambda\subset\Sigma$. Then $\Phi(\W_s,\Lambda)=\Phi(\W_s,t\Lambda)$, hence $\Phi(\W_s,\W_r)=\Phi(\W_s,t\W_r)$. These half-spaces correspond to markings $\mu=((s,t),r)$ with $K_\mu=\{t\}$ and $\mu'=(s,r)$ with $K_{\mu'}=\{r\}$. 
Consequently, the unique small component $\{t\}$ of $S\setminus(s\cup s^{\perp})$ is compatible with a big component. In view of Lemma~\ref{lem:big comp}, all the components are compatible. It remains to consider the case $\Lambda\not \subset\Sigma$.

Let $\tau_s$ (resp.\ $\tau_t$) be the elementary twist that sends $t$ to $sts$ (resp.\ $s$ to $tst$) and fixes other elements of $S$. 
For any $w\in\langle s,t \rangle$, composing appropriately $\tau_s$ and $\tau_t$ (while keeping the notation $s,t$ for the images of $s,t$ under the twist), we obtain $\tau=\tau(w)$ sending $s$ to $wsw^{-1}$, $t$ to $wtw^{-1}$ and fixing other elements of $S$. We will justify the following.
\begin{enumerate}
\item $\W_{\tau(s)}=w\W_s$ and $\W_{\tau(t)}=w\W_t$;
	\item if $J$ is maximal spherical that is exposed (resp.\ not exposed), then $D_{\tau(J)}=w.D_J$ (resp.\ $D_{\tau(J)}=D_J$);
	\item if $J$ and $I$ are both exposed (resp.\ not exposed), then $E_{\tau(J),\tau(I)}=w.E_{J,I}$ (resp.\ $E_{\tau(J),\tau(I)}=E_{J,I}$);
	\item if $J$ is exposed and $I$ is not exposed, then $E_{\tau(J),\tau(I)}=w.E_{J,I}$ and $E_{\tau(I),\tau(J)}=E_{I,J}$.
\end{enumerate}
Here (1) is immediate and implies (2), while (3) follows from (2) and Definition~\ref{def:E} (note that an elementary twist does change the defining graph, so it does not change the good subsets of $J$ and $I$). Now we prove (4). Note that for each $j\in J$, we have $\W_j\cap\W_{\tau(j)}\neq\emptyset$. Moreover, $\tau$ fixes each element of $I$. Thus for non-adjacent $i\in I$ and $j\in J$, the walls $\W_j$ and $\W_{\tau(j)}$ are in the same half-space for $i=\tau(i)$. Hence it follows from Definition~\ref{def:E} that $E_{\tau(I),\tau(J)}=E_{I,J}$. It remains to verify the first equality of (4). Note that the elements of $J\setminus\{s,t\}$ are fixed by $\tau$, and $\{s,t\}\subset J$ is maximal irreducible that is not good in view of Definition~\ref{defn:type I} and Lemma~\ref{double}. Thus $E_{\tau(J),\tau(I)}=D_{\tau(J)}=w.D_J=w.E_{J,I}$, finishing the proof of (4). 

Coming back to the case $\Lambda\not\subset\Sigma$, choose $\tau=\tau(w)$ a composition of twists as above so that $w\Sigma$ contains $\Lambda$. We will reach a contradiction by showing $\K_1(\tau(S))=\K_1(S)$ and $\K_2(\tau(S))<\K_2(S)$. The equality follows from the fact that for any maximal spherical $J\subset S$, we have $C_{\tau(J)}=C_J$. Now we verify the inequality. Consider maximal spherical subsets $J,I
\subset S$. If both $J$ and $I$ are exposed or both are not exposed, then by (3) we have $d(E_{\tau(J),\tau(I)},E_{\tau(I),\tau(J)})=d(E_{J,I},E_{I,J}).$ 

Now we assume that $J$ is exposed but $I$ is not exposed. Let $\beta$ be a shortest gallery from a chamber $y\in E_{I,J}$ to a chamber $x\in E_{J,I}$. 
By angle-compatibility, $\{s,t\}$ is conjugate to $\{s',t'\}\subset S'$. 
By Theorem~\ref{lem:residue}, we can assume that $\beta$ is a concatenation of galleries $\beta'$ and $\beta''$, where $\beta'$ is a minimal gallery from $y$ to some chamber (call it $x'$) in the $\{s',t'\}$-residue $\mathcal R$ containing $x$. Furthermore, $\beta'\subset\Lambda$. Note that $x\neq x'$ since $\Lambda\nsubseteq\Sigma$. 

We have $x'=w.x$ or $x'=w.x_{\mathrm{ant}}$, where $x_{\mathrm{ant}}$ is the chamber antipodal to $x$ in~$\mathcal R$. Note that $x_{\mathrm{ant}}\in E_{J,I}$, since $\{s,t\}$ is an irreducible component of $J$ that is not good with respect to $I$. Thus from (4) we deduce $x'\in E_{\tau(J),\tau(I)}$ and $y\in E_{\tau(I),\tau(J)}$. Consequently $d(E_{\tau(J),\tau(I)},E_{\tau(I),\tau(J)})<d(E_{J,I},E_{I,J})$, giving $\K_2(\tau(S))<\K_2(S)$.
\end{proof}

\subsection{Non-exposed small components} To prove Theorem~\ref{thm:minimal complexity},
it remains to consider the case where all components of $S\setminus(s\cup s^{\perp})$ are big, or small and not exposed.  
We argue by contradiction and assume that the components of $S\setminus(s\cup s^{\perp})$ can be divided into two non-empty families $\{A_k\}$ and $\{B_k\}$ such that all $\Phi_{A_k}$ coincide and are distinct from all $\Phi_{B_k}$, which also coincide. Let $A$ (resp.\ $B$) be the union of all $B_k$ (resp.\ $A_k$).  By Lemma~\ref{lem:big comp}, we can assume that all the big components (if they exist) are in $A$. 
Let $\tau$ be the elementary twist that sends each element $b\in B$ to $sbs$ and fixes other elements of $S$. By the proof of Lemma~\ref{lem:big comp}, we have $\K_1(S)=\K_1(\tau(S))$. For a contradiction, we will prove $\K_2(\tau(S))<\K_2(S)$.

Let $J\subset S$ be a maximal spherical subset. $J$ is \emph{twisted} if it contains an element of $B$. In that case, $s$ is adjacent to each element in $J$ since $B$ is a union of small components. Consequently $J\cup\{s\}$ is spherical so $s\in J$ by the maximality of $J$. 

Consider maximal spherical subsets $J$ and $I$. If both of them are twisted or both are not-twisted, then we have 
\begin{equation}
\label{eq:equal}
d(E_{\tau(J),\tau(I)},E_{\tau(I),\tau(J)})=d(E_{J,I},E_{I,J}).
\end{equation}
Now we assume that $J$ is twisted and $I$ is not twisted. If $I\subseteq \{s\}\cup \{s\}^{\perp}$, then
\eqref{eq:equal} holds as well. It remains to discuss the case where $I\nsubseteq s\cup s^{\perp}$. We will prove $d(E_{\tau(J),\tau(I)},E_{\tau(I),\tau(J)})<d(E_{J,I},E_{I,J})$, which implies $\K_2(\tau(S_0))<\K_2(S_0)$ and finishes the proof of
Theorem~\ref{thm:minimal complexity}.
 
\smallskip \noindent \textbf{Case 1: $I$ contains $s$}. 
In that case, pick $r\in I\setminus(s\cup s^{\perp})$. Let $I_1\subseteq I$ be maximal irreducible containing $r$. Then $s\in I_1$, since $s$ and $r$ do not commute. Pick $t\in J\setminus(s\cup s^{\perp})$. Let $J_1\subseteq J$ be maximal irreducible containing $t$. Then $s\in J_1$. 
Since both $t$ and $r$ are adjacent to $s$, we have that $t\in J_1$ is good with respect to~$r$, and $r\in I_1$ is good respect to~$t$. 

We first justify that $E_{J,I}$ and $E_{I,J}$ lie in distinct half-spaces for $s$. 
Otherwise, $\{r,s,t\}$ is geometric. In particular, we have $\Phi(\W_s,t\W_r)=\Phi(\W_s,r\W_t)$.
These half-spaces correspond to markings $\mu=((s,t),r)$ with $K_\mu=\{t\}$ and $\mu'=((s,r),t)$ with $K_{\mu'}=\{r\}$. This contradicts the assumption that $t$ and $r$ belong to incompatible components.

We have $D_{\tau(J)}=s.D_J$. Note that $\tau$ fixes all the elements of $I$ and $J\setminus J_1$, and hence $E_{\tau(J),\tau(I)}=s.E_{J,I}$
in view of  $$\Phi(s\W_t,\W_r)=\Phi(s\W_t,\W_r\cap\W_s)=s\Phi(\W_t,\W_r\cap\W_s)=s\Phi(\W_t,\W_r).$$
On the other hand, we have $E_{\tau(I),\tau(J)}=E_{I,J}$, since $\W_j\cap\W_{\tau(j)}\neq\emptyset$ for each $j\in J$, and hence $\W_j$ and~$\W_{\tau(j)}$ are in the same half-space for $i=\tau(i)\in I$ not adjacent to $j$.

To conclude Case 1, pick a gallery $\beta$ of minimal length from $x\in E_{J,I}$ to $y\in E_{I,J}$. Since chambers $x$ and $y$ lie in distinct  half-spaces for $s$ and $x$ is incident to $\W_s$, we can assume that the first edge of $\beta$ is dual to $\W_s$ (Theorem~\ref{lem:residue}). Since $s.x\in s.E_{J,I}=E_{\tau(J),\tau(I)}$ and $y\in E_{I,J}=E_{\tau(I),\tau(J)}$, we have $d(E_{\tau(J),\tau(I)},E_{\tau(I),\tau(J)})<d(E_{J,I},E_{I,J})$, as desired.

\smallskip \noindent \textbf{Case 2: $I$ contains an element not adjacent to $s$}. Let this element be $r$. 
Let $t$ and $J_1$ be as in Case~1. Since $t$ is inside a non-exposed small component, $t\in J_1$ is good with respect to $r$. In particular, $J_1$ is good with respect to $I$. 

Let $\Sigma\subset \Da$ be the union of the two sectors of the form $\Phi_s\cap \Phi_{t}$ for $\{\Phi_s, \Phi_{t}\}$ geometric. We first justify $\W_r\subset s\Sigma$.
Indeed, note that $\W_r$ is disjoint from any wall in $\W_{\{s,t\}}$. Since $s$ and $r$ are in the same component of $S\setminus(t\cup t^{\perp})$, we have $(t,r)\sim ((t,s),r)$ by Proposition~\ref{lem:compactible1} and the fact that $S$ is 2-rigid. Thus $\Phi(\W_t,\W_r)=\Phi(\W_t,s\W_r)$. It follows that 
$\W_r\subset \Sigma\cup s\Sigma$. Now recall that $t\in B$ and $r\in A$, thus $\Phi(\W_s,\W_r)\neq\Phi(\W_s,t\W_r)$ by the incompatibility of $A$ and $B$. It follows that $\W_r\subset\Sigma$ is not possible, justifying $\W_r\subset s\Sigma$.

Let $\Lambda$ be the sector of $\Sigma$ satisfying $\W_r\subset s\Lambda$. It follows that $E_{J,I}\subset \Lambda$ and $E_{\tau(J),\tau(I)}\subset s\Lambda$. Consequently $E_{\tau(J),\tau(I)}=sE_{J,I}$. We also have $E_{\tau(I),\tau(J)}=E_{I,J}$ as in Case~1. Note that $E_{I,J}$ and $E_{J,I}$ are in distinct half-spaces for $s$. Now we can prove $d(E_{\tau(J),\tau(I)},E_{\tau(I),\tau(J)})<d(E_{J,I},E_{I,J})$ in the same way as in Case 1.

\bibliography{1}
\bibliographystyle{plain}

\end{document}